\documentclass[12pt]{article}
\usepackage[left=2.8cm,right=2.8cm,top=2.2cm,bottom=2.2cm]{geometry}
\usepackage{amsfonts}
\usepackage{titlesec}
\usepackage{cite,color,xcolor}
\usepackage{amsmath}
\usepackage{amssymb}
\usepackage{times}
\usepackage{bm}
\usepackage[colorlinks,citecolor=blue,urlcolor=blue]{hyperref}
\usepackage[utf8]{inputenc}
\usepackage{mathtools}
\usepackage{amsthm}
\usepackage{setspace}
\usepackage[numbers]{natbib}
\usepackage{cases}
\usepackage{fancyhdr}
\usepackage[pagewise]{lineno}

\newtheorem{theorem}{Theorem}[section]

\newtheorem{lemma}{Lemma}[section]
\newtheorem{proposition}{Proposition}[section]

\newtheorem{definition}{Definition}[section]
\newtheorem{remark}{Remark}[section]

\newcommand{\bal}{\begin{align}}
\newcommand{\bbal}{\begin{align*}}
\newcommand{\beq}{\begin{equation}}
\newcommand{\eeq}{\end{equation}}
\newcommand{\bca}{\begin{cases}}
\newcommand{\eca}{\end{cases}}
\def\div{\mathord{{\rm div}}}
\newcommand{\pa}{\partial}
\newcommand{\fr}{\frac}

\newcommand{\dd}{\mathrm{d}}

\newcommand{\R}{\mathbb{R}}

\newcommand{\les}{\lesssim}

\newcommand\f{\left}
\newcommand\g{\right}

\begin{document}
\title{Ill-posedness of the hyperbolic Keller-Segel model in Besov spaces }

\author{Xiang Fei, Yanghai Yu and Mingwen Fei\\
\small School of Mathematics and Statistics, Anhui Normal University, Wuhu 241002, China\footnote{E-mail: fx19970912@163.com(X. Fei); yuyanghai214@sina.com(Y. Yu); mwfei@ahnu.edu.cn(M. Fei)}
}

\date{\today}

\maketitle\noindent{\hrulefill}

{\bf Abstract:} In this paper, we give a new construction of $u_0\in B^s_{p,\infty}$ such that the
corresponding solution to the hyperbolic Keller-Segel model starting from $u_0$ is
discontinuous at $t = 0$ in the metric of $B^s_{p,\infty}(\R^d)$ with $s>1+\fr{d}{p}$, $d\geq1$ and $1\leq p\leq\infty$, which implies the ill-posedness for this equation in $B^s_{p,\infty}$. Our result generalizes the recent work in \cite{Zhang01} (J. Differ. Equ. 334 (2022)) where the case $d=1$ and $p=2$ was considered.


\vskip0mm\noindent{\hrulefill}

\section{Introduction}
   Chemotaxis is the active motion of organisms influenced by
chemical gradients. The most prominent model for this process goes back to Patlak, Keller and Segel \cite{Patlak01, Keller01,Keller02, Keller03}
which takes the form of
\begin{equation}\label{ame}
       \begin{cases}
        \partial_{t}u+ \nabla \cdot(D_{1}(u,S)\nabla u-\chi(u,S)\nabla S)=0,\\
        \tau S_{t}=D_{2} \Delta S+k(u,S),
       \end{cases}
     \end{equation}
here $u(x,t)$ represents the cell density at position $x \in \mathbb{R}^{d}$, time $t > 0$, and $S(x,t)$ is the concentration of a chemical signal. The motility $D_{1}(u, S)$ and the chemotactic sensitivity $\chi(u, S)$
rely on the cell density and on the signal concentration. The term $k(u, S)$ depicts production
and decay or consumption of the signal and $D_{2}$ is the diffusion constant for $S$. The parameter
$\tau$ illustrates that movement of the species and dynamics of the signal have different characteristic time scales. The Keller-Segel model has been applied to many different problems, ranging from bacteria
chemotaxis to cancer growth or the immune response.

Dolak and Schmeiser \cite{Dolak02} derived a convection equation with a small diffusion term as higher order correction
from a kinetic model for chemotaxis. Inspired by this, Dolak and Schmeiser
proposed the following parabolic-type Keller-Segel equations with small diffusivity:
\begin{equation}\label{wme}
       \begin{cases}
        \partial_{t}u=- \nabla \cdot(u(1-u)\nabla S-\epsilon\nabla u), &\text{in}\quad \mathbb{R}^{+}\times\mathbb{R}^{d},\\
       -\Delta S=u - S,                           &\text{in}\quad \mathbb{R}^{+}\times\mathbb{R}^{d}.      \\
       \end{cases}
     \end{equation}
Burger, Dolak and Schmeiser \cite{Burger02} studied the asymptotic behavior of solutions of the chemotaxis model \eqref{wme} in multiple spatial dimensions. Of particular interest is the practically relevant case of small diffusivity, where (as in the one-dimensional case) the cell densities form
plateau-like solutions for large time. Some other results related to \eqref{wme} can be found in \cite{Tello01,Winkler01,Winkler02}.

Nie and Yuan \cite{Nie01} considered the Cauchy problem for multidimensional chemotaxis system
      \begin{equation}\label{bme}
       \begin{cases}
        \partial_{t}u-\Delta u=\div(uv), &\text{in}\quad \mathbb{R}^{+}\times\mathbb{R}^{d},\\
       \partial_{t}v-\nabla u=0,                           &\text{in}\quad \mathbb{R}^{+}\times\mathbb{R}^{d},      \:\   \\
        (u,v)|_{t=0}=(u_{0},v_{0}),              &\text{in}\quad \mathbb{R}^{d},
       \end{cases}
     \end{equation}
they proved that {(\ref{bme})} is well-posed in $\dot{B}_{p,\sigma}^{\frac{d}{p}-2}\times (\dot{B}_{p,\sigma}^{\frac{d}{p}-1})^{d}$ when $p<2d$ and is ill-posed when $p>2d$. Later  Nie and Yuan\cite{Nie02} also obtaind that {(\ref{bme})} is ill-posed in $\dot{B}_{p,1}^{\frac{d}{p}-2}\times (\dot{B}_{p,1}^{\frac{d}{p}-1})^{d}$ when $p=2d$. Almost in the same time,
 Xiao and Fei\cite{Xiao01} proved that {(\ref{bme})} is ill-posed in $\dot{B}_{p,\sigma}^{\frac{d}{p}-2}\times (\dot{B}_{p,\sigma}^{\frac{d}{p}-1})^{d}$ when $p=2d, \sigma>2$. Recently, Li, Yu and Zhu\cite{Li03} proved that  {(\ref{bme})} is ill-posed in $\dot{B}_{p,r}^{\frac{d}{p}-2}\times (\dot{B}_{p,r}^{\frac{d}{p}-1})^{d}$ when $1\leq r<d$.

 In this paper, we consider the Cauchy problem for following hyperbolic Keller-Segel equation:
      \begin{equation}\label{ame}
       \begin{cases}
        \partial_{t}u=- \nabla \cdot(u(1-u)\nabla S), &\text{in}\quad \mathbb{R}^{+}\times\mathbb{R}^{d},\\
       -\Delta S=u - S ,                          &\text{in}\quad \mathbb{R}^{+}\times\mathbb{R}^{d},      \\
        u(x,0)=u_{0}(x) ,               &\text{in}\quad \mathbb{R}^{d}.
       \end{cases}
     \end{equation}
The unknown scale functions $u(x, t)$ and $S(x, t)$ denote the cell density and the concentration of  chemical substance, respectively.
Dolak and  Schmeiser  \cite{Dolak01} firstly
established the existence and unique of global smooth solution to one dimensional version of \eqref{wme} with suitable
conditions on the initial data. On a time scale characteristic for the convective effects, they also
proved that the corresponding sequence of solutions $u^\epsilon$ converges to the weak entropy solution
$u$ to \eqref{ame} as $\epsilon\to0$.  Laterly, Burger, Difrancesco and Dolak\cite{Burger01} obtained the unique local-in-time solution to \eqref{wme} with the initial data belonging to $L^1(\R^d)\cap L^\infty(\R^d)$. Perthame and Dalibard \cite{Perthame01}  proved the existence of an entropy solution to \eqref{ame} by passing to the limit in a sequence of solutions
to the parabolic approximation. Lee and Liu\cite{Lee01} proved the sub-threshold for finite time shock formation to solutions of \eqref{ame} in one-dimension.

   Recently, Zhou, Zhang and Mu\cite{Zhou01} obtained the existence and uniqueness of solution of \eqref{ame} in ${{B}}_{p,r}^{s}(\mathbb{R}^{d})$ when $1\leq p,r\leq\infty,s>1+\frac{d}{p}$. Later,  Zhang, Mu and Zhou \cite{Zhang01} proved that \eqref{ame} is ill-posed in ${{B}}_{2,\infty}^{s}(\mathbb{R})$ with $s>\frac{3}{2}$ and \eqref{ame} is local well-posed in ${{B}}_{p,1}^{s}(\mathbb{R}^{d})$ when $1\leq p<\infty,
s=1+\frac{d}{p}$. However, their initial data seems to be valid only for $p=2$ when proving the ill-posedness in ${B}_{p,\infty}^{s}$.
Motivated by the recent works in \cite{Li01,Li02}, we aim to extend the ill-posedness result in \cite{Zhang01} to more general case, i.e, $1\leq p\leq\infty$ and $d\geq1$.
The main result of the paper is the following theorem:
\begin{theorem}\label{d1}
Let $d\geq1$. Assume that $$s>1+\frac{d}{p}\quad\text{with}\quad1\leq p\leq\infty,$$
then there exists $u_{0}\in{B}_{p,\infty}^{s}(\mathbb{R}^{d})$ and a positive constant $\epsilon_{0}$ such that the data-to-solution map $u_{0}\mapsto u$ of the Cauchy problem \eqref{ame} satisfies
\begin{align*}
\limsup_{t\to0^+}\|u(t)-u_{0}\|_{{B}_{p,\infty}^{s}}\geq \epsilon_{0}.
    \end{align*}
\end{theorem}

\begin{remark}
Theorem \ref{d1} demonstrates the ill-posedness of \eqref{ame} in ${B}_{p,\infty}^{s}(\mathbb{R}^{d})$. Precisely speaking, we can construct $u_{0}\in{B}_{p,\infty}^{s}(\mathbb{R}^{d})$ such that the corresponding solutions of the Keller-Segel equation do not converge to $u_{0}$ in the metric of ${B}_{p,\infty}^{s}(\mathbb{R}^{d})$ as $t\to0^+$.
\end{remark}
\begin{remark}
We should mention that the key
decomposition technique and the special initial data used in \cite{Zhang01} can not be applied to the present case $p\neq2$ any more. To overcome these difficulties, we construct a new initial data which is completely different from \cite{Zhang01}. In particular,  by utilizing the commutator estimate and some basic analysis, we make the proof more simple.
\end{remark}
The rest of the paper is organized as follows. In section 2, we introduce some basic definitions and key lemmas. In section 3, we present the proof of Theorem \ref{d1}.

 \section{Preliminaries}
{\bf Notation}\;
The notation $A\les B$ (resp., $A \gtrsim B$) means that there exists a harmless positive constant $c$ such that $A \leq cB$ (resp., $A \geq cB$).
Given a Banach space $X$, we denote its norm by $\|\cdot\|_{X}$. For a Banach space $X$ and for any $0<T\leq\infty$, we use standard notation $L^{p}(0,T;X)$ to denote the quasi-Banach space of Bochner measurable functions $f$ from $(0,T)$ to X endowed with the norm
 \begin{equation*}
 \|f\|_{{L^{p}_{T}}X}:=
 \begin{cases}
        \f(\int_0^T \|f(\cdot,t)\|^{p}_{X}\,dt\g)^{\frac{1}{p}},&\text{if $1\leq p<\infty$},\\
        \sup_{0\leq t \leq T}\|f(\cdot,t)\|_X,&\text{if $p=\infty$}.
 \end{cases}
 \end{equation*}
 Let us recall that for all $f\in \mathcal{S}'$, the Fourier transform $\widehat{f}$, is defined by
$$
(\mathcal{F} f)(\xi)=\widehat{f}(\xi)=\int_{\R^d}e^{-ix\cdot\xi}f(x)\dd x \quad\text{for any}\; \xi\in\R.
$$
 The inverse Fourier transform of any $g$ is given by
$$
(\mathcal{F}^{-1} g)(x)=\frac{1}{(2 \pi)^d} \int_{\R^d} e^{i x \cdot\xi}g(\xi)  \dd \xi.
$$
Next, we recall some facts on the Littlewood-Paley theory which can be found in \cite{B}.

Let $\varphi\in C_c^{\infty}(\mathbb{R}^d)$ and $\chi\in C_c^{\infty}(\mathbb{R}^d)$ be a radial positive function such that
    \begin{align*}
    &\mathrm{supp}\ \varphi\subset\{\xi\in \mathbb{R}^d:\frac34\leq|\xi|\leq\frac83\},  \:\: \:\: \mathrm{supp}\ \chi\subset\{\xi\in \mathbb{R}^d:|\xi|\leq\frac43\},\quad
    \\& \chi(\xi)+ \sum_{j\geq0}\varphi(2^{-j}\xi)=1\ \text{for any}\ \xi\in\mathbb{R}^d,
     \\&|i-j|\geq2\Rightarrow \mathrm{supp}\ \varphi(2^{-i}\cdot)\cap\mathrm{supp}\ \varphi(2^{-j}\cdot)=\varnothing,
     \\&j\geq1\Rightarrow \mathrm{supp}\ \varphi(2^{-j}\cdot)\cap\mathrm{supp}\ \chi(x)=\varnothing,\\
     &\varphi(\xi)\equiv 1\quad \text{for}\quad\frac43\leq |\xi|\leq \frac32.
      \end{align*}
We can  define the nonhomogeneous localization operators as follows.
\begin{align*}
         & \Delta_ju=0,\:\:j\leq-1;\:\:\:\:  \Delta_ju=\chi(D)u,\:\:j=-1;\:\:\:\:  \Delta_ju=\varphi(2^{-j}D)u,\:\:j\geq0,
    \end{align*}
where the pseudo-differential operator $f(D):u\to\mathcal{F}^{-1}(f \mathcal{F}u)$.

Let us now define the Besov spaces as follows.
\begin{definition}[\cite{B}]
Let $s\in\mathbb{R}$ and $(p,r)\in[1, \infty]^2$. The nonhomogeneous Besov space $B^{s}_{p,r}(\R^d)$ is defined by
\begin{align*}
B^{s}_{p,r}(\R):=\Big\{f\in \mathcal{S}'(\R^d):\;\|f\|_{B^{s}_{p,r}(\R^d)}<\infty\Big\},
\end{align*}
where
\begin{numcases}{\|f\|_{B^{s}_{p,r}(\R^d)}=}
\left(\sum_{j\geq-1}2^{sjr}\|\Delta_jf\|^r_{L^p(\R^d)}\right)^{\fr1r}, &if $1\leq r<\infty$,\nonumber\\
\sup_{j\geq-1}2^{sj}\|\Delta_jf\|_{L^p(\R^d)}, &if $r=\infty$.\nonumber
\end{numcases}
\end{definition}
\begin{remark}\label{re3}
It should be emphasized that $B^{s}_{p,\infty}(\R^d)$ with $s>\frac{d}{p}$ is a Banach algebra and $B^s_{p,\infty}(\R^d)\hookrightarrow B^t_{p,\infty}(\R^d)$ with $s>t$. These facts will be often used implicity.
\end{remark}

Finally, we recall some lemmas which be used later.

\begin{lemma}[Bernstein's inequality,\cite{B}]\label{pro3.7}
 Let $\mathcal{C}$ be an annulus and $\mathcal{B}$ be a ball. There exists a constant $C$ such that for any nonnegative integer $k$, any couple $(p,q)\in[1,\infty]^{2}$ with $1\leq p\leq q$, and any $L^{p}$ function $u$ we have
 \begin{align*}
&\sup_{|\alpha|=k}\|\partial^{\alpha}u\|_{L^{q}(\mathbb{R}^{d})}\leq C^{k+1}\lambda^{k+d(\frac{1}{p}-\frac{1}{q})}\|u\|_{L^{p}(\mathbb{R}^{d})}, \:\: \mathrm{supp} \widehat u\subset\lambda \mathcal{B},
\\&C^{-(k+1)}\lambda^{k}\|u\|_{L^{p}(\mathbb{R}^{d})}\leq \sup_{|\alpha|=k}\|\partial^{\alpha}u\|_{L^{p}(\mathbb{R}^{d})}\leq C^{(k+1)}\lambda^{k}\|u\|_{L^{p}(\mathbb{R}^{d})},\:\: \mathrm{supp}\widehat u\subset\lambda \mathcal{C}.
    \end{align*}
\end{lemma}
\begin{lemma}[\cite{B}]\label{pro3.5}
A smooth function $f:\mathbb{R}^{d}\rightarrow \mathbb{R}$ is said to be an $S^{m}$-multiplier: if\:\:$\forall\alpha\in \mathbb{N}^{d}$, there exists a constant $C_{\alpha}>0$ such that
 \begin{align*}
|\partial^{\alpha}f(\xi)|\leq C_{\alpha}(1+|\xi|)^{m-\alpha},\:\:\xi\in\mathbb{R}^{d}.
    \end{align*}
If $f$ is a $S^{m}$-multiplier, then the operator $f(D)$ is continuous from ${B}_{p,r}^{s}$ to ${B}_{p,r}^{s-m}$ for all $s\in \mathbb{R}$ and $1\leq p,r\leq\infty$.
\end{lemma}
\begin{lemma}[\cite{B}]\label{pro3.6}
For $1\leq p\leq\infty$ and $s>0$, there exists a constant $C$,depending continuously on $p$ and $s$, we have
 \begin{align*}
&\big\|2^{js}\|[\Delta_{j},v]\cdot\nabla f\|_{L^{p}}\big\|_{\ell^{\infty}}\leq C(\|\nabla v\|_{L^{\infty}}\|f\|_{{B}_{p,\infty}^{s}}+\|\nabla f\|_{L^{\infty}}\|\nabla v\|_{{ B}_{p,\infty}^{s-1}}),
\end{align*}
where $[\Delta_{j},v]\cdot\nabla f=\Delta_{j}(v\cdot\nabla f)-v\cdot\Delta_{j}\nabla f$.
\end{lemma}
\section{Proof of Theorem \ref{d1}}
For convenience of computation, we rewrite \eqref{ame} as follows
\begin{equation}\label{dme}
       \begin{cases}
        \partial_{t}u+(1-2u)\nabla S\cdot\nabla u+u(1-u)\Delta S=0, &\text{in}\quad \mathbb{R}^{+}\times\mathbb{R}^{d}, \\
        S=(1-\Delta)^{-1}u,          &\text{in}\quad    \mathbb{R}^{+}\times\mathbb{R}^{d}, \\
        u(x,0)=u_{0}(x)               , &\text{in}\quad \mathbb{R}^{d}.
       \end{cases}
     \end{equation}
Let $\widehat{\phi}\in C_{0}^{\infty}(\mathbb{R})$ be an even, real-valued and nonnegative function which satisfies
\begin{numcases}
 {\widehat{\phi}(\xi)=}
        1,& if $|\xi|\leq\frac{1}{4^{d}}$,\nonumber\\
        0,& if $|\xi|\geq\frac{1}{2^{d}}$.\nonumber
\end{numcases}
\begin{remark}\label{re5} By the  Fourier-Plancherel formula, we have $\phi(x)=\mathcal{F}^{-1}(\widehat{\phi}(\xi))$.
It is easy to check that
\begin{align*}
&\phi(0)=\fr{1}{2\pi}\int_{\R}\widehat{\phi}(\xi)\dd \xi>0\quad\text{and}\quad
\phi'(0)=\fr{1}{2\pi}\int_{\R}\widehat{\phi'}(\xi)\dd \xi=0.
    \end{align*}
\end{remark}
\begin{lemma}\label{pro3.1}
Define  the function $f_{n}(x)$ by
\begin{align*}
f_{n}(x)=\phi(x_{1})\sin\f(\frac{17}{12}2^{n}x_{1}\g)\phi(x_{2})\cdot\cdot\cdot\phi(x_{d}),\quad n\geq3.
    \end{align*}
Then
\begin{numcases}
{\Delta_{j}(f_{n})=}
f_{j}, &if $j=n$,\nonumber\\
0, &if $j\neq n$.\nonumber
\end{numcases}
\end{lemma}
\begin{proof} Notice that
\begin{align*}
\mathrm{supp} \ \widehat{f}_{n}&\subset \Big\{\xi\in\R^d: \ \frac{17}{12}2^{n}-\fr12\leq |\xi|\leq \frac{17}{12}2^{n}+\fr12\Big\},
\end{align*}
using the definition of $\Delta_{j}$ enables us to get the desired result. For more details see \cite{Li4}.
\end{proof}
\begin{proposition}\label{pro3.2}
Define  the initial data $u_0(x)$ as
\begin{align*}
&S_{0}(x):=\sum_{n=3}^{\infty}2^{-n(s+2)}f_{n}(x),\\
&u_{0}(x):=(1-\Delta)S_{0}(x).
    \end{align*}
If $s>1+\frac{d}{p}$, we have
\begin{align*}
 \|u_{0}\|_{{B}_{p,\infty}^{s}}\leq C.
    \end{align*}
\end{proposition}
\begin{proof}
By Lemma \ref{pro3.1}, we have
 \begin{align*}
 \Delta_{j}S_{0}=2^{-j(s+2)}f_{j}(x).
    \end{align*}
Combining the above and Minkowski's inequality yields
\begin{align*}
 \|u_{0}\|_{{B}_{p,\infty}^{s}}&\leq C\|S_{0}\|_{{B}_{p,\infty}^{s+2}}=\sup_{j\geq 0}2^{(s+2) j}\|\Delta_{j}S_0\|_{L^p}\leq C.
    \end{align*}
    We complete the proof of Proposition \ref{pro3.2}.
\end{proof}

Using Proposition \ref{pro3.2} and Theorem 1.1 in \cite{Zhou01}, we can obtain that there exists a short time $T>0$ that {(\ref{dme})} has a  unique solution
$u\in  L^{\infty}([0,T);{B}_{p,\infty}^{s})\cap Lip([0,T);{B}_{p,\infty}^{s-1})$ for $s>1+\frac{d}{p}$. Moreover, it  holds
\begin{align}\label{FYF}
 \|u(t)\|_{L^{\infty}_{T}({B}_{p,\infty}^{s})}\leq C\|u_{0}\|_{{B}_{p,\infty}^{s}}.
    \end{align}

\begin{proposition}\label{pro3.3}
Let $s-1>\frac{d}{p}$ and $\|u_{0}\|_{{B}_{p,\infty}^{s}}\les1$. Assume that $u\in L^{\infty}(0,T;{B}_{p,\infty}^{s}(\mathbb{R}^{d}))$ be the  solution of \eqref{ame}, then we have
 \begin{align*}
 \|u(t)-u_{0}\|_{{B}_{p,\infty}^{s-1}} \les t.
    \end{align*}
\end{proposition}
\begin{proof}
Using  the Newton-Leibniz formula, Minkowski's inequality, Remark \ref{re3}, Lemma \ref{pro3.5} and Proposition \ref{pro3.2}, we have
\begin{align*}
 \|u(t)-u_{0}\|_{{B}_{p,\infty}^{s-1}}&\leq \int_0^t \|(1-2u)\nabla S\cdot\nabla u\|_{{B}_{p,\infty}^{s-1}}\dd\tau+\int_0^t \|u(1-u)\Delta S\|_{{ B}_{p,\infty}^{s-1}}\dd\tau
 \\&\les t\|u\|_{L_{t}^{\infty}{B}_{p,\infty}^{s-2}}\|u\|_{L_{t}^{\infty}{B}_{p,\infty}^{s}}+t\|u\|_{L_{t}^{\infty}{ B}_{p,\infty}^{s-2}}\|u\|_{L_{t}^{\infty}{B}_{p,\infty}^{s-1}}\|u\|_{L_{t}^{\infty}{B}_{p,\infty}^{s}}\\&\quad+t\|u\|_{L_{t}^{\infty}{ B}_{p,\infty}^{s-1}}^{3}+t\|u\|_{L_{t}^{\infty}{B}_{p,\infty}^{s-1}}^{2}
 \\&\les t\f(\|u\|_{L_{t}^{\infty}{B}_{p,\infty}^{s}}^{3}+\|u\|_{L_{t}^{\infty}{B}_{p,\infty}^{s}}^{2}\g)
 \\&\les t\f(\|u_{0}\|_{{B}_{p,\infty}^{s}}^{3}+\|u_0\|_{{B}_{p,\infty}^{s}}^{2}\g)\\
 &\les t,
    \end{align*}
 where we have used \eqref{FYF}.

We complete the proof of Proposition \ref{pro3.3}.
\end{proof}

\begin{proposition}\label{pro3.4}
Let $s-1>\frac{d}{p}$ and $\|u_{0}\|_{{B}_{p,\infty}^{s}}\les1$. Assume that $u\in L^{\infty}(0,T;{B}_{p,\infty}^{s}(\mathbb{R}^{d}))$ be the  solution of \eqref{ame}, then we have
 \begin{align*}
 \|h(t,u_{0})\|_{{B}_{p,\infty}^{s-2}}\les t^{2},
    \end{align*}
    where we denote $$h(t,u_{0}):=u-u_{0}+tv_{0}$$ and $$v_{0}:=\nabla \cdot(u_0(1-u_0)\nabla S_0)=(1-2u_{0})\nabla S_{0}\cdot\nabla u_{0}+u_{0}(1-u_{0})\Delta S_{0}.$$
\end{proposition}
\begin{proof}
Using  the Newton-Leibniz formula, Minkowski's inequality, Remark \ref{re3}, Lemma \ref{pro3.5} and \eqref{FYF}, we have
\begin{align*}
 \|h(t,u_{0})\|_{{B}_{p,\infty}^{s-2}}&\leq \int_0^t \|\partial_{\tau}u+v_{0}\|_{{B}_{p,\infty}^{s-2}}\dd\tau
 \\&\leq \int_0^t \|\nabla \cdot(u_0(1-u_0)\nabla S_0)-\nabla \cdot(u(1-u)\nabla S)\|_{{B}_{p,\infty}^{s-2}}\dd\tau
 \\&\les\int_0^t \|u_0(1-u_0)\nabla S_0-u(1-u)\nabla S\|_{{B}_{p,\infty}^{s-1}}\dd\tau\\
 &\les \int_0^t \|u(\tau)-u_{0}\|_{{B}_{p,\infty}^{s-1}}\dd\tau\\
 &\les\int_0^t \tau\dd\tau\\
 &\les t^2,
    \end{align*}
 where we have used Proposition \ref{pro3.3} in the last step.

We complete the proof of Proposition \ref{pro3.4}.
\end{proof}

Now we present the proof of Theorem \ref{d1}.\\
{\bf Proof of Theorem \ref{d1}.}\quad Notice that
$
u(t)-u_0=h(t,u_{0})-tv_0,
$
then
 \begin{align}\label{eme}
\nonumber\|u(t)-u_{0}\|_{{B}_{p,\infty}^{s}}&\geq2^{js}\|\Delta_{j}(h(t,u_{0})-tv_{0})\|_{L^{p}}
\nonumber\\&\geq 2^{js}t\|\Delta_{j}v_{0}\|_{L^{p}}-2^{js}\|\Delta_{j}h(t,u_{0})\|_{L^{p}}
\nonumber\\&\geq2^{js}t\|\Delta_{j}((1-2u_{0})\nabla S_{0}\cdot\nabla u_{0})\|_{L^{p}}
-2^{js}t\|\Delta_{j}(u_{0}^{2}\Delta S_{0})\|_{L^{p}}
\nonumber\\&
\quad-2^{js}t\|\Delta_{j}(u_{0}\partial_{x}^{2} S_{0})\|_{L^{p}}
-2^{js}\|\Delta_{j}h(t,u_{0})\|_{L^{p}}.
 \end{align}

It is not difficult to deduce that
\begin{align*}
&2^{js}t\|\Delta_{j}(u_{0}^{2}\Delta S_{0})\|_{L^{p}}\les t\|u_{0}^{2}\Delta S_{0}\|_{{B}_{p,\infty}^{s}}
\les t\|u_{0}\|_{{B}_{p,\infty}^{s}}^{3}\les t,\\
&2^{js}t\|\Delta_{j}(u_{0}\Delta S_{0})\|_{L^{p}}\les t\|u_{0}\Delta S_{0}\|_{{B}_{p,\infty}^{s}}
\les t\|u_{0}\|_{{B}_{p,\infty}^{s}}^{2}\les t
,\\
&2^{js}\|\Delta_{j}h(t,u_{0})\|_{L^{p}}\leq 2^{2j}\|h(t,u_{0})\|_{{B}_{p,\infty}^{s-2}}
\les t^{2}2^{2j}.
 \end{align*}

Gathering the above estimates together with {(\ref{eme})} yields
\begin{align}\label{z}
\nonumber\|u(t)-u_{0}\|_{{B}_{p,\infty}^{s}}&\geq2^{js}\|\Delta_{j}(h+tv_{0})\|_{L^{2}}
\nonumber\\&\geq 2^{js}t\|\Delta_{j}((1-2u_{0})\nabla S_{0}\cdot\nabla u_{0})\|_{L^{p}}
-Ct-Ct^{2}2^{2j}
\nonumber\\&\geq2^{js}t\|(1-2u_{0})\nabla S_{0}\cdot\Delta_{j}\nabla u_{0}\|_{L^{p}}\nonumber\\&
\quad-2^{js}t\|[\Delta_{j},(1-2u_{0})\nabla S_{0}]\cdot\nabla u_{0}\|_{L^{p}}-Ct-Ct^{2}2^{2j}.
\end{align}

On the one hand, by Lemma \ref{pro3.6}, we deduce
\begin{align}\label{z1}
2^{js}\|[\Delta_{j},(1-2u_{0})\nabla S_{0}]\cdot\nabla u_{0}\|_{L^{p}}\leq C.
\end{align}

On the other hand, we have
\begin{align}\label{ime}
2^{js}\|(1-2u_{0})(\nabla S_{0}\cdot\Delta_{j}\nabla u_{0})\|_{L^{p}}&= 2^{js}\f\|(1-2u_{0})\sum_{i=1}^{d}\partial_{x_{i}}S_{0}\Delta_{j}\partial_{x_{i}}u_{0}\g\|_{L^{p}}\geq J-K,
 \end{align}
 where
\begin{align*}
&J:=2^{js}\f\|(1-2u_{0})\partial_{x_{1}}S_{0}\Delta_{j}\partial_{x_{1}}u_{0}\g\|_{L^{p}},\\
&K:=2^{js}\sum_{i=2}^{d}\f\|(1-2u_{0})\partial_{x_{i}}S_{0}\Delta_{j}\partial_{x_{i}}u_{0}\g\|_{L^{p}}.
 \end{align*}

By Lemma \ref{pro3.1}, we infer
\begin{align}\label{nme}
J&=2^{-2j}\f\|(1-2u_{0})\partial_{x_{1}}S_{0}\partial_{x_{1}}(1-\Delta)f_{j}\g\|_{L^{p}}
\geq 2^{-2j}\f(J_{1}-J_{2}-J_{3}\g),
  \end{align}
  where
  \begin{align*}
&J_{1}:=\f\|(1-2u_{0})\partial_{x_{1}}S_{0}\partial^3_{x_{1}}f_{j}\g\|_{L^{p}},\\
&J_{2}:=\sum_{i=2}^d\f\|(1-2u_{0})\partial_{x_{1}}S_{0}\partial_{x_{1}}\pa^2_{x_i}f_{j}\g\|_{L^{p}},\\
&J_{3}:=\f\|(1-2u_{0})\partial_{x_{1}}S_{0}\partial_{x_{1}}f_{j}\g\|_{L^{p}}.
 \end{align*}

  We have
 \begin{align*}
\partial^3_{x_{1}}f_{j}(x)=-\f(\frac{17}{12}\g)^32^{3j}\phi(x_{1})\cos\f(\frac{17}{12}2^{j}x_{1}\g)\phi(x_{2})\cdot\cdot\cdot\phi(x_{d})+R,
    \end{align*}
 where
 \begin{align*}
R
&=\frac{17}{4}2^{j}\phi^{\prime\prime}(x_{1})\cos\f(\frac{17}{12}2^{j}x_{1}\g)\phi(x_{2})\cdot\cdot\cdot\phi(x_{d})\\
&
\quad-3\f(\frac{17}{12}\g)^22^{2j}\phi^{\prime}(x_{1})\sin\f(\frac{17}{12}2^{j}x_{1}\g)\phi(x_{2})\cdot\cdot\cdot\phi(x_{d})\\
&\quad+\phi^{\prime\prime\prime}(x_{1})\sin\f(\frac{17}{12}2^{j}x_{1}\g)\phi(x_{2})\cdot\cdot\cdot\phi(x_{d}).
    \end{align*}

 Obviously, $\f\|(1-2u_{0})\partial_{x_{1}}S_{0}R\g\|_{L^{p}}\leq C2^{2j}$, then
    \begin{align*}
J_{1}&\geq\f(\frac{17}{12}\g)^32^{3j}\f\|(1-2u_{0})\partial_{x_{1}}S_{0}\phi(x_{1})\cos\f(\frac{17}{12}2^{j}x_{1}\g)\phi(x_{2})\cdot\cdot\cdot\phi(x_{d})\g\|_{L^{p}}-C2^{2j}.
   \end{align*}

 By the construction of $f_n$, it is not difficult to deduce that
$$S_{0}(0)=\sum_{n=3}^{\infty}2^{-n(s+2)}f_{n}(0)=0.$$

By easy computations, we have
\begin{align*}
\Delta S_{0}(x)& =\partial_{x_{1}}^{2}S_{0}(x)+\sum_{i=2}^{d}\partial_{x_{i}}^{2}S_{0}(x)\\
&=\sum_{n=3}^{\infty}2^{-n(s+2)}\partial_{x_{1}}^{2}f_{n}+\sum_{i=2}^{d}\sum_{n=3}^{\infty}2^{-n(s+2)}\partial_{x_{i}}^{2} f_{n}.
 \end{align*}
 Noticing that the construction of $f_n$ again and using the fact $\phi'(0)=0$ from Remark \ref{re5}, then we obtain
 $$\Delta S_{0}(0)=0,$$
 which implies that
$$u_{0}(0)=S_{0}(0)-\Delta S_{0}(0)=0.$$

Since $(1-2u_{0})\partial_{x_{1}}S_{0}\phi(x_{1})\phi(x_{2})\cdot\cdot\cdot\phi(x_{d})$ is a real-valued and continuous function on $\R$, then there exists some $\delta>0$ such that for any $x\in B_{\delta}(0):=\{x\in\mathbb{R}^{d}:\;|x|\leq\delta\}$
\begin{align*}
&\f|[(1-2u_{0})\partial_{x_{1}}S_{0}\phi(x_{1})\phi(x_{2})\cdot\cdot\cdot\phi(x_{d})](x)\g|\nonumber\\
\geq&\; \fr{1}{2}\f|[(1-2u_{0})\partial_{x_{1}}S_{0}\phi(x_{1})\phi(x_{2})\cdot\cdot\cdot\phi(x_{d})](0)\g|\nonumber\\
=&\;\fr{1}{2}\phi^{d}(0)|\partial_{x_{1}}S_{0}(0)|\nonumber\\
=&\;\fr{17}{24}\phi^{2d}(0)\sum\limits^{\infty}_{n=3}2^{-n(s+1)}=:c_0>0.
\end{align*}
Thus we have for $j$ large enough
\begin{align*}
J_{1}&\geq c_02^{3j}\f\|\cos\f(\frac{17}{12}2^{j}x_{1}\g)\g\|_{L^{p}(B_{\delta}(0))}-C2^{2j}
\geq \tilde{c}_02^{3j}.
   \end{align*}

By direct computations, we can verify that
\begin{align*}
J_{2}+J_{3}\leq C2^{j}.
 \end{align*}

Thus, we have
 \begin{align}\label{jme}
J\geq C2^{j}.
 \end{align}

Similarly, we also have
 \begin{align}\label{kme}
K\leq C.
 \end{align}

Combining \eqref{jme} and \eqref{kme}, we have
\begin{align}\label{lme}
2^{js}t\|(1-2u_{0})\nabla S_{0}\cdot\Delta_{j}\nabla u_{0}\|_{L^{p}}&\geq C2^{j}t.
 \end{align}

Inserting \eqref{lme} and \eqref{z1} into \eqref{z}, we deduce that for large $j$
\begin{align*}
\|u(t)-u_{0}\|_{{B}_{p,\infty}^{s}}\geq C2^{j}t-Ct-C2^{2j}t^{2}\geq C2^{j}t-C2^{2j}t^{2}.
 \end{align*}

Thus, picking $t2^{j}\approx\epsilon_0$ with small $\epsilon_0$, we have
\begin{align*}
\|u(t)-u_0\|_{B^{s}_{p,\infty}}\geq C\epsilon_0-C\epsilon_0^2\geq c_1\epsilon_0.
\end{align*}
This completes the proof of Theorem \ref{d1}.

\vspace*{1em}
\noindent\textbf{Acknowledgments} The authors would like to express their gratitude to the anonymous referees for valuable suggestions and comments which greatly improved the paper.
Y. Yu is supported by NSF of China under Grant No. 12101011. M. Fei is supported by NSF of China under Grants No. 11871075 and 11971357.

\vspace*{1em}
\noindent\textbf{Data Availability} No data was used for the research described in the article.

\vspace*{1em}
\noindent\textbf{Conflict of interest}
The authors declare that they have no conflict of interest.

\end{document}